\documentclass[a4paper,10pt]{amsart}
\usepackage{amsmath,amstext,amsthm,amscd,amsopn,verbatim,amssymb}
\usepackage{hyperref}

\usepackage[margin=1in]{geometry}
\usepackage{tikz}
\usetikzlibrary{matrix}
\usetikzlibrary{shapes}
\usetikzlibrary{arrows}
\tikzset{ext/.style={circle, draw,inner sep=1pt},int/.style={circle,draw,fill,inner sep=1pt},nil/.style={inner sep=1pt}}
\tikzset{exte/.style={circle, draw,inner sep=3pt},inte/.style={circle,draw,fill,inner sep=3pt}}
\tikzset{diagram/.style={matrix of math nodes, row sep=3em, column sep=2.5em, text height=1.5ex, text depth=0.25ex}}
\tikzset{diagram2/.style={matrix of math nodes, row sep=0.5em, column sep=0.5em, text height=1.5ex, text depth=0.25ex}}
\theoremstyle{plain}
  \newtheorem{thm}{Theorem}
  \newtheorem{defi}{Definition}
  
  \newtheorem{prop}{Proposition}
  
  \newtheorem{lemma}{Lemma}
\theoremstyle{definition}
  
  \newtheorem*{rem}{Remark}

\newcommand{\p}{\partial}

\newcommand{\C}{{\mathbb{C}}}
\newcommand{\R}{{\mathbb{R}}}
\newcommand{\bbH}{{\mathbb{H}}}

\newcommand{\bpm}{\begin{pmatrix}}
\newcommand{\epm}{\end{pmatrix}}

\newcommand{\Tpoly}{ T_{\rm poly} }
\newcommand{\Dpoly}{ D_{\rm poly} }

\newcommand{\mU}{\mathcal{U}}

\newcommand{\supp}{{\mathrm{supp}}}
\newcommand{\gstar}{{\mathrm{star}}}

\newcommand{\lca}{{\mathrm{lca}}}
\newcommand{\ocd}{{\mathrm{gcd}}}
\newcommand{\tops}{{\mathit{top}}}

\begin{document}
\title{Logarithms and deformation quantization}
\author{Anton Alekseev}
\address{Section of Mathematics\\ University of Geneva \\ 2-4 rue du Li\`evre, CP 64, 1211 Gen\`eve 4, Switzerland}
\email{Anton.Alekseev@unige.ch}


\author{Carlo A. Rossi}
\address{Zurich, Switzerland}
\email{mquve.odessa@gmail.com}


\author{Charles Torossian}
\address{Institut Math\'ematiques de Jussieu-Paris rive gauche, Universit\'e Paris Diderot,
UFR de math\'ematiques Case 7012, 75205 Paris Cedex 13} \email{torossian@math.univ-paris-diderot.fr}

\author{Thomas Willwacher}
\address{Department of Mathematics\\ University of Zurich\\ Winterthurerstrasse 190 \\ 8057 Zurich, Switzerland}
\email{thomas.willwacher@math.uzh.ch}


\begin{abstract}
We prove the statement/conjecture of M. Kontsevich on the existence of the logarithmic formality morphism $\mU^{{\rm log}}$. This question was open since 1999,  and the main obstacle was the presence of $dr/r$ type singularities near the boundary $r=0$ in the integrals over compactified configuration spaces.  The novelty of our approach is the use of local torus actions on configuration spaces of points in the upper half-plane. It gives rise to a version of Stokes' formula for differential forms with singularities at the boundary which implies the formality property of $\mU^{{\rm log}}$. 

We also show that the logarithmic formality morphism admits a globalization from $\R^d$ to an arbitrary smooth manifold.
\end{abstract}

\maketitle

\section{Introduction}

The deformation quantization problem for smooth manifolds was solved by M. Kontsevich in \cite{K1}. The solution is a corollary of the Formality Theorem which asserts the existence of an $L_\infty$-quasi-isomorphism
\[
 \mU\colon \Tpoly M \to \Dpoly M
\]
from the graded Lie algebra of multivector fields $\Tpoly M$ to the differential graded Lie algebra of multi-differential operators $\Dpoly M$ on a smooth manifold $M$. The main part of the argument is an explicit formula for the morphism $\mU$ in the case of $M=\R^d$ with coefficients 
defined as integrals over compactified configuration spaces of points in the upper half-plane,
\[
 \varpi_\Gamma = \int_{{\rm Conf}_{n,m}} \Omega_\Gamma .
\]
Here ${\rm Conf}_{n,m}$ is the compactified configuration space of $n$ points in the upper-half plane and $m$ points on the real line, its dimension is given by formula ${\rm dim}({\rm Conf}_{n,m}) = 2n+m-2$, $\Gamma$ is a graph with $n+m$ vertices, the vertices of $\Gamma$ are in bijection $i \mapsto z_i$ with $n+m$ points of the configuration. The form $\Omega_\Gamma$ in the integrand is defined as a product of one-forms, one for each edge of $\Gamma$,
\begin{equation}  \label{eq:1}
 \Omega_\Gamma = \prod_{(i,j)\in E\Gamma}\,  \frac{1}{2\pi} \, d\arg\left( \frac{z_i-z_j}{\bar z_i-z_j}\right) .
\end{equation}
Hence, the degree of $\Omega_\Gamma$ is equal to $|E \Gamma|$, and the weights $\varpi_\Gamma$ are well-defined if $|E \Gamma| = 2n+m-2 = {\rm dim}({\rm Conf}_{n,m})$.

The statement that $\mU$ is an $L_\infty$-morphism translates into a set of quadratic identities 
\begin{equation}   \label{intro:quadratic}
\sum_i \,  \varpi_{\Gamma'_{i}} \varpi_{\Gamma''_{i}} =0
\end{equation}
for the weights. Here , $\Gamma$ is a graph with $|E\Gamma|= 2n+m -3$, $\Gamma'_i \subset \Gamma$ is a subgraph of $\Gamma$, and $\Gamma''_i$ is obtained from $\Gamma$ by contracting $\Gamma'_i$. It turns out that quadratic equations \eqref{intro:quadratic} can be obtained by applying the Stokes' Theorem to the form $\Omega_\Gamma$,
\begin{equation} \label{intro:stokes}
0=\int_{{\rm Conf}_{n,m}} d\Omega_\Gamma
 =
 \int_{\partial {\rm Conf}_{n,m}} \Omega_\Gamma
 =
 \sum_{i} \int_{\partial_i {\rm Conf}_{n,m}} \Omega_\Gamma .
\end{equation}
Here the boundary components $\partial_i {\rm Conf}_{n,m}$ factorize as products of configuration spaces with smaller $n,m$, and so does the integrand $\Omega_\Gamma$ when restricted to the boundary. Hence, the right-hand side can be re-written as
$$
 \sum_{i} \int_{\partial_i {\rm Conf}_{n,m}} \Omega_\Gamma = \sum_i \,  \varpi_{\Gamma'_{i}} \varpi_{\Gamma''_{i}} .
$$
 
In \cite[section 4.1]{K2}, M. Kontsevich stated/conjectured that one can construct another $L_\infty$-morphism between $\Tpoly M$ and $\Dpoly M$ by replacing the form \eqref{eq:1} with the {\em logarithmic form}
\begin{equation}  \label{eq:2}
 \Omega_\Gamma^{{\rm log}} = \prod_{(i,j)\in E\Gamma} \frac{1}{2\pi i} d\log\left( \frac{z_i-z_j}{\bar z_i-z_j}\right) .
\end{equation}
In the path integral approach of Cattaneo-Felder \cite{CF}, this will correspond to a different gauge choice for the Poisson $\sigma$-model. Referring to the physics language of Feynman calculus, one often refers to the 1-form factors on the right-hand side of \eqref{eq:2} as to {\em logarithmic propagators}.

The idea is elegant, and it should lead to an $L_\infty$-morphism with better number theoretic properties, but its realization encounters a number of technical difficulties. First of all, the logarithmic forms $\Omega_\Gamma^{{\rm log}}$ do not extend to compactified configuration spaces, in general. Hence, {\em a priori} it is not clear whether they give rise to well-defined {\em logarithmic weights}
\begin{equation}  \label{intro:weight}
 \varpi_\Gamma^{{\rm log}} = \int_{{\rm Conf}_{n,m}} \Omega_\Gamma^{{\rm log}} .
 \end{equation}
It is even more problematic to apply the Stokes' formula \eqref{intro:stokes} to forms with singularities on the boundary. 

The purpose of this paper is to prove M. Kontsevich's statement/conjecture on the existence of the logarithmic formality morphism. Our first result is Theorem \ref{thm:formregular} stating that for $|E\Gamma| = 2n+m-2$, the logarithmic forms $\Omega_\Gamma^{{\rm log}}$ extend to regular forms on compactified configuration spaces ${\rm Conf}_{n,m}$. Hence, the logarithmic weights \eqref{intro:weight} are actually well-defined.

In the case of $|E\Gamma| = 2n+m-3$ which is relevant for the Stokes' formula, the forms $\Omega_\Gamma^{{\rm log}}$ do possess singularities at the boundary of the configuration space. Our second result is Theorem \ref{thm:regStokes} which provides a version of the (regularized) Stokes' formula for differential forms with well controlled singularities suitable for our purposes. Then, Theorem \ref{thm:vanishingproperty} shows that logarithmic forms $\Omega_\Gamma^{{\rm log}}$ are exactly of this type and computes the boundary contributions in the regularized Stokes' formula. This result includes a version of the Kontsevich Vanishing Lemma for logarithmic weights. Finally, Theorem \ref{thm:final} states the $L_\infty$-property for the new formality morphism $\mU^{{\rm log}}$.

The main novelty of our approach comes from the observation that  configuration spaces carry local torus actions. For a set of points $z_1, \dots, z_k$ collapsing to their center of mass $\zeta$ in the upper half-plane, we define a circle action
$$
z_i \mapsto e^{i\theta} (z_i - \zeta) + \zeta
$$
for $i=1, \dots, k$. Other points of the configuration are not affected by this action. It is convenient to introduce a coordinate $r \geq 0$ near the boundary of the configuration space such that 
$$
z_i = \zeta + r z_i^{(1)},
$$
where the new coordinates $z_i^{(1)}$ are normalized in such a way that $\sum_i z_i^{(1)}=0, \sum_i |z_i^{(1)}|^2 =1$. We show that for $r$ small the logarithmic forms admit the following decomposition
$$
\Omega_\Gamma^{{\rm log}} = \frac{dr}{r} \wedge \alpha + {\rm terms \,\, regular\,\, in \,\,} r,
$$
where $\alpha$ is basic for the $S^1$-action. In particular, one can define the {\em regularization} of the logarithmic form at the boundary
$$
{\rm Reg}(\Omega_\Gamma^{{\rm log}}) = \left( \Omega_\Gamma^{{\rm log}} - \frac{dr}{r} \wedge \alpha \right)_{r=0} .
$$
These forms enter  the regularized Stokes' formula
$$
\int_{{\rm Conf}_{n,m}} d \Omega_\Gamma^{{\rm log}} = \sum_i \, \int_{\partial_i {\rm Conf}_{n,m}} \, {\rm Reg}_i (\Omega_\Gamma^{{\rm log}}) .
$$
which implies the quadratic relation needed to establish the $L_\infty$-property of $\mU^{{\rm log}}$.

Our result has already been used in the literature. In \cite{carloexplicit} the characteristic class of the star product defined by the formality morphism $\mU^{{\rm log}}$ is computed, and in the forthcoming paper \cite{carlothomas} our methods are used to construct a family of Drinfeld associators interpolating between the Knizhnik-Zamolodchikov and the Alekseev-Torossian associators.

The structure of the paper is as follows. In Section \ref{sec:Stokes}, we prove a version of the Stokes' formula for differential forms with singularities. In Section \ref{sec:Config}, we introduce special charts and study local torus actions on configuration spaces. In Section \ref{sec:log}, we consider logarithmic forms and show that they satisfy the assumptions of the regularized Stokes' formula. In Section \ref{sec:final}, we show that this Stokes' formula leads to a new $L_\infty$-morphism $\mU^{{\rm log}}$ and show that $\mU^{{\rm log}}$ admits a globalization to arbitrary smooth manifolds $M$.

\vskip 0.3cm

{\bf Authorship.} This project required an effort of several researchers. It was started by C.T. who came up with a version of Theorem \ref{thm:formregular}.  This work was continued by C.R who understood the importance of the regularization procedure in Stokes' formula and proved the globalization property of $\mU^{{\rm log}}$. The idea of using the local torus actions is due to T.W. The present version of the manuscript has been written up by A.A. and T.W.

\vskip 0.3cm

{\bf Acknowledgements.} We are grateful to G. Felder and S. Merkulov for inspiring discussions and for the interest in our work.
We would like to thank J. L\"offer who participated in the earlier stages of the project.
The work of A.A. was supported in part by the grant MODFLAT of the European Research Council (ERC) and by the grants
140985 and 141329 of the Swiss National Science Foundation.
T.W. was partially supported by the Swiss National Science Foundation, grant 200021\_150012.
Research of A.A. and T.W. was supported in part by the NCCR SwissMAP of the Swiss National Science Foundation.

\section{Regularized Stokes formula}   \label{sec:Stokes}

In this section, we prove a version of the Stokes' formula suitable for differential forms with a certain (well controlled) type of singularities. In view of applications to configuration spaces, we use the framework of manifolds with corners.

Let $K$ be a compact manifold with corners of dimension $n$ covered by a system of charts $U_i$
locally  diffeomorphic to open subsets in $\R_{\geq 0}^k \times \R^{n-k}$. We assume that 
the charts are labeled by the partially ordered set $I$ such that 

\begin{itemize}
\item
$U_i \cap U_j =\emptyset$ unless $i\geq j$ or $j\geq i$.
\item
$U_i$ carries a free action of a torus $T_i$ preserving all components of the boundary. If $i>j$ under the partial order, one has a natural injective group homomorphism $T_i \hookrightarrow T_j$ such that the inclusions $ U_i \cap U_j \hookrightarrow U_i, U_i \cap U_j \hookrightarrow U_j$ are $T_i$-equivariant. 
\end{itemize}

Furthermore, we assume that $K$ admits a partition of unity $\rho_i$ subordinate to the atlas $\{ U_i\}$ such that each function $\rho_i$ is invariant under the action of the torus $T_i$.

For each $i$,  we denote by $v_{i,a}$ the fundamental vector fields of the circles $S^1_{i,a} \subset T_i$ for $a = 1, \dots , {\rm dim}(T_i)$.
For future use it is convenient to introduce multi-vector fields  
\[
 \xi_i = \bigwedge_{a=1}^{\mathrm{dim} T_i} v_{i,a} .
\]
In particular, for $T_i$ trivial we have $\xi_i=1$.

\begin{defi}
For $\xi$ a multi-vector field and $\omega$ a differential form, we say that $\omega$ is $\xi$-basic if $\iota_\xi\omega=0$ and $\iota_\xi d\omega=0$. 
\end{defi}
For the 0-vector field $\xi=1$ every $\xi$-basic form vanishes.
If $\xi$ is a vector field generating an action of the group $\R$ (or $S^1$), we recover the notion of a basic form with
respect to this action. Indeed, we have $\iota_\xi \omega=0$ and 
$L_\xi \omega = d \iota_\xi \omega + \iota_\xi d \omega =0$.
For any multi-vector field  $\xi$, the $\xi$-basic forms form a subcomplex under the de Rham differential.

\begin{rem}
Let $\xi$ and $\xi'$ be two multi-vector fields and $\omega$ be a $\xi$-basic differential form. Then, it is also basic
with respect to the multi-vector field $\xi \wedge \xi'$. Indeed, 
$\iota_{\xi \wedge \xi'} \omega = \pm \iota_{\xi'} \iota_\xi \omega =0 $, and the same argument applies to 
$\iota_{\xi \wedge \xi'} d \omega$.

\end{rem}


\begin{defi}
 Let $\omega$ be a degree top-1 differential form on $K^\circ$. We say that $\omega$ is \emph{regularizable} if for every $j$ there is a $\xi_j$-basic form $\alpha_j$ (the \emph{counterterm}) defined on $U_j$ such that $\omega-\alpha_j$ is regular on the boundary $\p K\cap U_j$.
The \emph{regularization} ${\rm Reg}(\omega)$ of $\omega$ is the top-degree form on the boundary $\p K$ defined by
\[
{\rm Reg}(\omega)\mid_{\p K\cap U_j}=(\omega-\alpha_j)\mid_{\p K\cap U_j}.
\]
\end{defi}

\begin{prop}
 The regularisation is well-defined, i.~e. it does not depend on the chart and on the choice of counterterms.
Furthermore, the form $\iota_{\xi_j}\omega$ is regular on the boundary $\p K \cap U_j$ and 
\[
{\rm Reg}(\omega)\mid_{\p K\cap U_j} = \eta_j \wedge (\iota_{\xi_j}\omega)\mid_{\p K\cap U_j},
\]
where $\eta_j$ is a $(\mathrm{dim} T_j)$-form such that $\iota_{\xi_j}\eta_j=1$.
\end{prop}


\begin{proof}
Let $U_j$ and $U_k$ be two overlapping charts. 
Since $U_j \cap U_k =\emptyset$ unless $j\geq k$ or $j\leq k$, it suffices to show the statement for $j \geq k$. Then, $T_j \subset T_k$ and $\xi_k = \xi_j \wedge \xi'$ for some multi-vector field $\xi'$. Let $\eta_k$ be a differential form on $U_k$ such that $\iota_{\xi_k} \eta_k = 1$ and $\alpha_j, \alpha_k$ be counterterms on the charts $U_j$ and $U_k$, respectively. Then, on the interior of the overlap $U_j \cap U_k$ we have 
%
\[
\eta_k \wedge \iota_{\xi_k}(\omega-\alpha_j)=\eta_k \wedge \iota_{\xi_k}(\omega-\alpha_k)
\]
since $\iota_{\xi_k}\alpha_j = \iota_{\xi_k}\alpha_k=0$.
But the left-hand side and the right-hand side extend continuously to the boundary, so the identity extends to the boundary as well. On the boundary, the operator $\eta_k \wedge \iota_{\xi_k}$ is the identity on top degree forms, and we obtain
\[ 
(\omega - \alpha_j)\mid_{\p K \cap U_j \cap U_k} = (\omega -\alpha_k)\mid_{\p K \cap U_j \cap U_k} .
\]
Hence, regularizations on different charts agree. The same argument applies to a comparison of two different counterterms on the same chart.

For the second statement, we use  $\iota_{\xi_j} \alpha_j=0$ to show that 
$\iota_{\xi_j} \omega = \iota_{\xi_j} (\omega - \alpha_j)$ in the interior of $U_j$. Hence, $\iota_{\xi_j} \omega$ is regular at the boundary
$\p K \cap U_j$. Using the definition of the regularization, we obtain
\[ {\rm Reg}(\omega)\mid_{\p K \cap U_j} = (\omega -\alpha_j)\mid_{\p K \cap U_j} =
\eta_j \wedge \iota_{\xi_j} (\omega - \alpha_j)\mid_{\p K \cap U_j} =
\eta_j \wedge \iota_{\xi_j} \omega\mid_{\p K \cap U_j},
\]
as required.


%
\end{proof}

\begin{thm}[Regularized Stokes' Theorem]\label{thm:regStokes}
 Let $\omega$ be a regularizable top-1 degree form on $K$. Then, the differential form $d\omega$ is regular on $K$ and
\[
 \int_K d\omega = \int_{\p K} {\rm Reg}(\omega).
\]
\end{thm}
\begin{proof}
In the interior of $U_j$, $d\omega$ is a top degree form and we have 
\[
 d\omega = \eta_j\wedge \iota_{\xi_j} d\omega =
\eta_j\wedge \iota_{\xi_j} d(\omega-\alpha_j)
=
d(\omega-\alpha_j),
\]
where we have used that $\iota_{\xi_j} d \alpha_j =0$. Since the form $\omega - \alpha_j$ is regular
on the boundary, so is its differential $d(\omega - \alpha_j)$. Hence, $d\omega$ is regular on the 
boundary which proves the first statement.

We will now prove the Stokes' formula from right to left.
\begin{align*}
\int_{\p K}{\rm Reg}(\omega) &=
\sum_j \, \int_{\p K}\rho_j {\rm Reg}(\omega)
= 
\sum_j \, \int_{\p K\cap U_j}\rho_j {\rm Reg}(\omega)
= 
\sum_j \, \int_{\p K\cap U_j}\rho_j (\omega-\alpha_j)
\\
&=
\sum_j \int_{U_j}d(\rho_j (\omega-\alpha_j))
=
\sum_j \int_{U_j}d(\rho_j \omega)
=
 \int_{K}d\big( \big(\sum_j\rho_j\big) \omega\big)
=
\int_{K}d \omega
\end{align*}
Here we used that $\rho_j\alpha_j$ is $\xi_j$-basic to conclude that
\[
d(\rho_j \alpha_j)
=
\eta_j\wedge \iota_{\xi_j} (d\rho_j \wedge \alpha_j+\rho_j d\alpha_j)
=
\eta_j\wedge (\pm d\rho_j \wedge \iota_{\xi_j}\alpha_j+\rho_j \iota_{\xi_j}d\alpha_j)
=0.
\]
This implies that $d(\rho_j \omega)$ is regular. Since $d\omega$ is regular, both $\rho_j d\omega$ and $d\rho_j \wedge \omega$ are regular. Hence, all forms involved in the Stokes' formula argument are regular.
\end{proof}

\begin{rem}
In our setting the form ${\rm Reg}(\omega)$ will be $T_i$-invariant when restricted to $\p_i K = \p K \cap U_i$. In this case, the integral over the $i$-th boundary stratum (of codimenson 1) may be written as 
\[
\int_{\p_i K} {\rm Reg}(\omega) =\int_{\p_i K/T_i} \iota_{\xi_i}\omega,
\]
where we have used the normalization ${\rm Vol}(T_i) =1$.
\end{rem}

\section{Configuration spaces}   \label{sec:Config}

In this Section, we show that compactified configuration spaces of points in the upper half-plane satisfy the assumptions of the previous section. We will start by discussing the compactified configuration spaces of points in $\C$.

\subsection{Nested families}
Let $A$ be a finite set, and let $\mathbb{C}^A$ be the space of injections $z: A \to \C$ such that $z_a \neq z_b$ if $a \neq b$. Recall that the group $G_4= \C^* \ltimes \C$ acts freely on $\C^A$ if $|A| \geq 2$. This action is given by formula
$$
z_a \mapsto u z_a + v
$$
for all $a \in A$. The quotient ${\rm FM}_A = \C^A/G_4$ is a complex space of dimension ${\rm dim}_\C \, {\rm FM}(A) = |A| -2$. 

The action of $G_4$ restricts to a free action of its real subgroup $G_3 = \R_+ \ltimes \C$, where $u \in \R_+$ and $v \in \C$. The quotient  ${\rm Conf}^{{\rm open}}_A = \C^A/G_3$ is a real manifold of dimension ${\rm dim}_\R \, {\rm Conf}^{{\rm open}}_A=2|A|-3$. It is a principal $G_4/G_3 \cong S^1$-bundle over ${\rm FM}_A$. 
For every element $[z] \in {\rm Conf}^{{\rm open}}_A$ there is a unique representative $z \in \C^A$ such that
$$
\sum_{a \in A} z_a = 0 \hskip 0.3cm, \hskip 0.3cm
\sum_{a \in A} |z_a|^2 = 1.
$$
These equations define an ellipsoid $\mathcal{E}_{A,}$ in the space $\C^{|A|} = \R^{2|A|}$ and ${\rm Conf}^{{\rm open}}_A$ is identified with an open dense subset of this ellipsoid.

For any $A' \subset A$ there is a natural projection ${\rm Conf}^{{\rm open}}_A \to {\rm Conf}^{{\rm open}}_{A'}$ (by forgetting the points corresponding to elements outside $A'$). As a consequence, there is a natural map ${\rm Conf}^{{\rm open}}_A \to \mathcal{E}_{A'}$.
Hence, one obtains a map
$$
{\rm Conf}^{{\rm open}}_A \to \prod_{A' \subset A} \, \mathcal{E}_{A'},
$$
where $A'$ ranges over of all subsets of $A$ with $|A'| \geq 2$. This embedding defines the Kontesvich compactification of the configuration space denoted by ${\rm Conf}_A$. This compactification 
comes equipped with natural projections ${\rm Conf}_A \to {\rm Conf}_{A'}$ for all $A' \subset A$ of cardinality at least 2.

We will consider families of subsets $i \in 2^A$ such that $A \in i$ and $\{a\}\in i$ for all $a\in A$. We say that such a family is nested 
if for any two subsets $B, C \in i$ one of the following three options is realized: $B \cap C = \emptyset$, $B \subset C$ or $C \subset B$. Such a nested family corresponds to a rooted tree (denoted by the same letter) with the root $A$, and with direct descendants (children) of a set $B \in i$
given by the subsets $C \subset B$ such that $C \in i$ and there is no $D \in i$ such that $C \subset D \subset B$. The leafs of this tree are one-element subsets, i.~e. sets of the form $\{a\}$.
We denote the direct ancestor (the parent) of a set $B$ in the tree by $p(B)$ and the set of children of $B$ by $\gstar(B)$. The inclusion relation
defines a partial order on the elements of $i$, and we have $B \cap C = \emptyset$ for elements $B$ and $C$ which are not comparable 
with respect to this partial order. For two sets $B, C \in i$ we denote by ${\rm lca}(B,C)$ their least common ancestor in the corresponding tree. 

There is in turn a partial order on the nested families of subsets of $A$, given by the set theoretic inclusion. For nested families $i$ and $j$ we will say that $i\geq j$ if $i\subset j$. The largest nested family in this ordering (or smallest set theoretically) is the set $\tops$ consisting of $A$ and all $\{a\}\in A$. For two nested families $i$ and $j$, one can ask for the greatest common descendant ${\rm gcd}(i,j)$ under the partial order. A priori, the existence of ${\rm gcd}(i,j)$ is not guaranteed. If $i \cup j$ is a nested family, then ${\rm gcd}(i,j)= i \cup j$. And if $i \cup j$ is not a nested family, then ${\rm gcd}(i,j)$ does not exist.

For $i$ a nested family, we will set $\underline i=i\setminus (\{A\} \cup \{\{a\}\mid a\in A\})$ and $|i|:=|\underline i|$ abusing notation.

\subsection{Charts on configuration spaces}
For each nested family $i$ and a sufficiently small constant $c$, we introduce a subset $V_i^c$ of ${\rm Conf}_A$ which is defined as follows. 
If $i=\tops$ we set $V_i^c={\rm Conf}_A$ irrespective of $c$.
Otherwise, let $B \in i$ and denote
$$
\zeta_B = 
\frac{1}{|B|} \, \sum_{b \in B} \, z_b
$$
the coordinate of its center of mass.
In particular, $\zeta_{\{b\}}=z_b$. Then, a configuration $z \in \C^A$ is in $V_i^c$ if for every $B \in i$ we have
\begin{equation}   \label{eq:1/N}
\frac{|\zeta_C - \zeta_B|}{|\zeta_D - \zeta_B|} \leq c
\end{equation}
for all $C \in {\rm star}(B)$ and for all $D \in {\rm star}(p(B))\setminus\{B\}$.
It is clear that if $c<c'$ then $V_i^c\subset V_i^{c'}$.

Note that the left-hand side of \eqref{eq:1/N} is invariant under the action of $G_3$. Hence, the inequality still makes sense on the open configuration space ${\rm Conf}^{{\rm open}}_A$.  It extends in a natural way to the compactified configuration space ${\rm Conf}_A$. By abuse of notation, we denote the corresponding subsets of ${\rm Conf}_A$ by $V_i^c$ as well.

\begin{lemma}\label{lem:Vprops}
The sets $V_i^c$ satisfy the following properties:
\begin{enumerate}
\item Together, they cover ${\rm Conf}_A$.
\item For each $c>0$: $\partial_i {\rm Conf}_A \subset V_i^c$, where $\partial_i {\rm Conf}_A$ is the boundary stratum of the compactified configuration space corresponding to the nested set $i$.
\item\label{lemitem:Vprop3} There is a constant $c'$, tending to 0 as $c\to 0$, such that for all $B,C, D\in i$ with $\lca(B,C)<\lca(B,D)$:
\[
 \frac{|\zeta_C - \zeta_B|}{|\zeta_D - \zeta_B|} \leq c'
\]
for any configuration in $V_i^c$. 
\item\label{lemitem:Vprop4} For each $c$ sufficiently small there is a constant $c''$ tending to 0 as $c\to 0$ such that 
\[
V_i^c \cap V_j^c \subset  
V_{\gcd(i,j)}^{c''} .
\]
In case ${\rm gcd}(i,j)$ does not exist, the right-hand side is understood as the empty set.

\end{enumerate} 
\end{lemma}
\begin{proof}
The first two statements are obvious. 
For the third statement, the relevant part of the tree corresponding to $i$ looks schematically as follows:
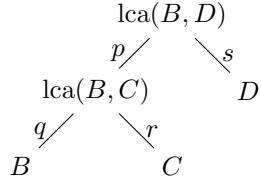
\[
 \begin{tikzpicture}
  \node(v) at (0,2) {$\lca(B,D)$};
  \node(D) at (1,1) {$D$};
  \node(w) at (-1,1) {$\lca(B,C)$};
  \node(B) at (-2,0) {$B$};
  \node(C) at (0,0) {$C$};
  \draw (v) edge node[right] {$s$} (D) edge node[left] {$p$} (w) (w) edge node[left]{$q$} (B) edge node[right]{$r$} (C);
 \end{tikzpicture}.
\]
Here the numbers $p,q,r,s$ shall indicate the number of levels between the nodes. Note that $q=0$ or $r=0$ is allowed.
We will do an induction on $p,q,r,s$. We will use a generic constant $c'$ tending to zero as 
$c\to 0$ in the induction hypothesis, the precise form of $c'$ as function of $c$ will not be kept track of.
First, one may perform an induction on $p+q$, reducing the statement to the case $p=1$, $q=0$. Concretely, we may estimate
\[
 |\zeta_B-\zeta_C|\leq |\zeta_{p(B)}-\zeta_C|+|\zeta_{p(B)}-\zeta_B| 
 \leq 2 c' |\zeta_B-\zeta_D|
\]
using the induction hypothesis.
Then, one performs an induction on $r+s$. One estimates 
\[
 |\zeta_B-\zeta_C|\leq |\zeta_B-\zeta_{p(C)}| + |\zeta_C-\zeta_{p(C)}| 
 \leq c'|\zeta_B-\zeta_D| + c |\zeta_S-\zeta_{p(C)}| 
 \leq (1+c)c' |\zeta_B-\zeta_D|,
\]
where $S$ is any sibling of $p(C)$ under $p(p(C))$ and we used the induction hypothesis twice. Similarly, 
\begin{align*}
 |\zeta_D-\zeta_B| &\geq |\zeta_{p(D)}-\zeta_B|-|\zeta_{p(D)}-\zeta_{D}|
 \geq
 |\zeta_{p(D)}-\zeta_B|-c|\zeta_{p(D)}-\zeta_{S'}|
 \geq
 |\zeta_{p(D)}-\zeta_B|-c c' |\zeta_{p(D)}-\zeta_{B}|
 \\&= |\zeta_{p(D)}-\zeta_B|(1-c c')
 \geq |\zeta_{C}-\zeta_B|\frac{1-c c'}{c'},
 \end{align*}
where $S'$ is a sibling of $p(D)$ and we again used the induction hypothesis.
This reduces the statement to the case $q=0$ and $p=r=s=1$ which is just \eqref{eq:1/N}.


Consider the last assertion of the lemma.
First assume that ${\rm gcd}(i,j)$ does not exist. Then, $i \cup j$ is not a nested family and one can choose $B\in i$, $C\in j$ such that the subsets $B \setminus C, B \cap C$ and $C \setminus B$ are non-empty. Pick elements $b_1\in B\setminus C$, $b_2\in B\cap C$, $b_3\in C\setminus B$.
Then, by the third assertion the two inequalities
\begin{align*}
 |z_{b_1}-z_{b_2}|&\leq c' |z_{b_2}-z_{b_3}|
 &
 |z_{b_2}-z_{b_3}|&\leq c' |z_{b_1}-z_{b_2}|
\end{align*}
have to be satisfied simultaneously, a contradiction for $c$ sufficiently small so that $c'<1$. Hence it follows that in this case $V_i^c \cap V_j^c=\emptyset$.

Finally, let's assume that $i$ and $j$ do have a common descendant $k:={\rm gcd}(i,j)$. We have to check that each defining inequality \eqref{eq:1/N} of $V_k^{c''}$ (for some $c''$ possibly slightly larger than $c$) is implied by the defining inequalities of $V_i^c$ and $V_j^c$. 
Let $B,C,D\in k$ be as in \eqref{eq:1/N}. Without loss of generality, we may assume that $B\in i$. If also $C\in i$, then by the third assertion above 
\[
 |\zeta_B-\zeta_C| \leq c' |\zeta_B-z_d|
\]
for all $d\in D$. Furthermore, again by the third assertion and irrespective of whether $D$ is in $i$ or $j$: $|\zeta_{D}-z_d|\leq c'|\zeta_B-z_b|$ for any $b\in B$. 
It follows that if $c$ is small enough, then for some slightly larger constant $c''$ (tending to 0 as $c\to 0$) 
\[
|\zeta_B-\zeta_C| \leq c'' |\zeta_B-\zeta_D|. 
\]
If $C$ is not in $i$, then by the same argument we still obtain an inequality 
\[
 |\zeta_B-z_x| \leq c'' |\zeta_B-z_D|
\]
for all $x\in C$. But 
\[
 |\zeta_B-\zeta_C|=|\zeta_B-\frac 1 {|C|} \sum_{x\in C} z_x|
 \leq 
 \frac 1 {|C|} \sum_{x\in C}|\zeta_B-\zeta_C|
 \leq c'' |\zeta_B-\zeta_D|
\]
and we are done.
\end{proof}

\subsection{Coordinates and torus actions}
The space ${\rm Conf}_A$ is a manifold with corners. This structure can be described in the following way. Let $i$ be a nested family and $V_i^c$ be the corresponding chart. For each subset $B \in i$, introduce a parameter $r_B \geq 0$. First, assume that $\underline i = \{ B\}$ and choose $b \in B$. Then, we use the parametrization
$$
z_b = \zeta_B + r_B z^{(1)}_b,
$$
where the new coordinates $z^{(1)}_b$ are normalized such that
\begin{align*}
 \sum_{b \in B} z^{(1)}_b &=0
 & &\text{and} &
 \sum_{b \in B}  \big|z^{(1)}_b\big|^2 &= 1.
\end{align*}
The boundary component defined by the equation $r_B=0$ is isomorphic to ${\rm Conf}^{{\rm open}}_{A \backslash B \cup \beta} \times {\rm Conf}^{{\rm open}}_B$, where 
$\{ z^{(1)}_b \}_{b \in B}$ are coordinates on ${\rm Conf}^{{\rm open}}_B$ and $\{ z_a \}_{a \in A \backslash B}$, $\zeta_B$ are coordinates on ${\rm Conf}^{{\rm open}}_{A \backslash B \cup \beta}$ where the new element $\beta$ is mapped to $\zeta_B$. 
 In general, one repeats this procedure for chains of embeddings $b \in B_k \subset B_{k-1} \subset \dots \subset B_1$ to obtain parametrizations of the form
$$
z_b = \zeta_{B_1} + r_{B_1}(\zeta^{(1)}_{B_2} + r_{B_2}( \dots (\zeta^{(k-1)}_{B_{k}} + r_{B_k}z^{(k)}_b) \dots ).
$$

Consider the set $V_i^c$ for $c$ sufficiently small and some $B\in i$ with $2\leq |B|< |A|$. We define an action of the circle group $S^1=:S^1_B$ by rotating all points in $B$ around their center of mass $\zeta_B$. For $c$ sufficiently small, this action is well-defined by assertion \ref{lemitem:Vprop3} of Lemma \ref{lem:Vprops}.

\begin{prop}
Let $i$ be a nested family and $B,C \in i$. Then, the actions of $S^1_B$ and $S^1_C$ commute.
\end{prop}

\begin{proof}
If $B \cap C = \emptyset$, the actions of $S^1_B$ and of $S^1_C$ move different groups of points while preserving their centers of mass. Hence, they commute.

If $B \subset C$, the action of $S^1_B$ on the  points corresponding to elements $c \in C \backslash B$ is trivial. Hence, it commutes with the action of $S^1_C$. For the points corresponding to elements $b \in B$, we obtain for a rotation around $\zeta_B$ followed by a rotation around $\zeta_C$
$$
\begin{array}{lll}
z & \mapsto & e^{i \theta_B} (z - \zeta_B) + \zeta_B \\
& \mapsto & e^{i\theta_C}(e^{i \theta_B} (z - \zeta_B) + \zeta_B - \zeta_C) + \zeta_C \\
& = & e^{i(\theta_B + \theta_C)} z + e^{i\theta_C}(1-e^{i\theta_B}) \zeta_B + (1-e^{i\theta_C}) \zeta_C.
\end{array}
$$
And for a rotation around $\zeta_C$ followed by a rotation around $\zeta'_B=e^{i\theta_C}(\zeta_B-\zeta_C) + \zeta_C$ we get
$$
\begin{array}{lll}
z & \mapsto & e^{i \theta_C} (z - \zeta_C) + \zeta_C \\
& \mapsto & e^{i\theta_B}(e^{i \theta_C} (z - \zeta_C) + \zeta_C - e^{i\theta_C}(\zeta_B-\zeta_C) - \zeta_C) + e^{i\theta_C}(\zeta_B-\zeta_C) + \zeta_C \\
& = & e^{i(\theta_B + \theta_C)} z + e^{i\theta_C}(1-e^{i\theta_B}) \zeta_B + (1-e^{i\theta_C}) \zeta_C,
\end{array}
$$
as required.
\end{proof}

Hence, on the the set $V_i^c$ for $c$ sufficiently small we obtain an action of the torus $T_i$ of dimension ${\rm dim} \, T_i = |i|$. Note that the coordinate functions $r_B \geq 0$ for $B \in i$ are invariant under the action of $T_i$. For the circle action $S^1_B$, we shall denote the corresponding fundamental vector field by $v_B$.

\subsection{The partition of unity}

In this section we shall use the subsets $V_i^c$ to construct the partition of unity on the configuration space with the property that the function $\rho_i$ is $T_i$-invariant.

We call nested families $i$, $j$ \emph{non-ancestors} if they are incomparable in the partial ordering, i.~e., it neither holds that $i\geq j$ nor that $i \leq j$. If  $\ocd(i,j)$ exists this is equivalent to saying that $\ocd(i,j)<i,j$.

Fix sufficiently small numbers $0<c_0<\tilde c_0<c_1<\tilde c_1<\cdots $ such that all torus actions are defined and such that for all non-ancestors $i, j$:
\begin{equation}\label{equ:intersectionprop}
V_i^{\tilde c_{|i|}} \cap V_j^{\tilde c_{|j|}} \subset  
V_{{\ocd}(i,j)}^{c_{|{\ocd}(i,j)|}}.
\end{equation}
Here we again interpret the right-hand side as the empty set if $\ocd(i,j)$ does not exist.
Concretely we can choose such $c_n, \tilde c_n$ by an (inverse) recursion on $n$.
At the $n$-th stage one chooses $\tilde c_n$ such that for all non-ancestors $i,j$ with $|i|\leq |j|=n$
\[
 V_i^{\tilde c_{n}} \cap V_j^{\tilde c_{n}} \subset  
V_{\ocd(i,j)}^{c_{|\ocd(i,j)|}}.
\]
This is possible since there are only finitely many such $i,j$ and by assertion \ref{lemitem:Vprop4} of Lemma \ref{lem:Vprops}. Note also that automatically $|\ocd(i,j)|>n$ (if $\ocd(i,j)$ exists) and hence $c_{|\ocd(i,j)|}$ is already known from earlier stages of the recursion.
Then one picks $c_n<\tilde c_n$ arbitrarily. The required inequality then follows since if $n=\max(|i|, |j|)$ then $\tilde c_{|i|},\tilde c_{|j|}\leq \tilde c_n$ and hence:
\[
V_i^{\tilde c_{|i|}} \cap V_j^{\tilde c_{|j|}} 
\subset 
V_i^{\tilde c_{n}} \cap V_j^{\tilde c_{n}}.
\]

Choose functions $\chi_i$ on ${\rm Conf}_A$ such that
\begin{itemize}
\item $\chi_i\equiv 1$ on a neighborhood of $V_i^{c_{|i|}}$.
\item $\chi_i$ is supported on $V_i^{\tilde c_{|i|}}$.
\item $\chi_i$ is invariant under the $T_i$ action.
\end{itemize}
In particular we choose $\chi_{\tops}\equiv 1$.
Then we define a partition of unity $\rho_i$ recursively such that
\[
\rho_i = \chi_i (1-\sum_{j, |j|>|i|}\rho_j).
\]

\begin{lemma}\label{lem:rhoiprops}
The functions $\rho_i$ thus defined are indeed a partition of unity, i.~e., $0\leq \rho_i\leq 1$ and $\sum_i\rho_i=1$. Furthermore $\supp\rho_i\cap \supp \rho_j=\emptyset$ if $i$ and $j$ are non-ancestors.
\end{lemma}
To prove the Lemma, we will show by (reverse) induction on $n$ that the following statements hold true  
\begin{enumerate}
\item\label{itm:inditem1}  $0\leq \rho_i\leq 1$ for all $i$ with $|i|\geq n$
\item\label{itm:inditem2}  $\sum_{k,|k|\geq n}\rho_k\leq 1$.
\item\label{itm:inditem3} $\sum_{k,|k|\geq n}\rho_k = 1$ on a neighborhood of each $V_i^{c_{|i|}}$, with $n\leq |i|$.
\item\label{itm:inditem4} If $i$ and $j$, $|i|$, $|j|\geq n$ are non-ancestors, then $\supp \rho_i\cap \supp \rho_j=\emptyset$ 
\end{enumerate}
\begin{proof}
Indeed, one verifies that these properties hold for the highest degree, using assertion \ref{lemitem:Vprop4} of Lemma \ref{lem:Vprops}.
Now suppose the above properties hold for $n+1$, and we want to show them for $n$. 
Assertion \ref{itm:inditem1} (for $n$) follows immediately from the induction assumption \ref{itm:inditem2} (for $n+1$).
Furthermore, suppose $|j|\geq |i|= n$, and that $i$ and $j$ are non-ancestors. Then, since $|\ocd(i,j)|>\max(n,|j|)$ and by the induction assumption \ref{itm:inditem3} above
\begin{multline*}
\{z\mid \sum_{k,|k|>|i|}\rho_k(z)<1\}=
 \{z\mid \sum_{k,|k|>n}\rho_k(z)<1\}
 \subset 
 \{z\mid \sum_{k,|k|>|j|}\rho_k(z)<1\}
 \\
 \subset 
 \{z\mid \sum_{k,|k|\geq |\ocd(i,j)|}\rho_k(z)<1\}
 \subset 
 \left(\tilde V_{{\rm gcd}(i,j)}^{c_{|{\rm gcd}(i,j)|}}\right)^c  
\end{multline*}
where $\tilde V_{{\rm gcd}(i,j)}^{c_{|{\rm gcd}(i,j)|}}$ is a neighborhood of $V_{{\rm gcd}(i,j)}^{c_{|{\rm gcd}(i,j)|}}$ as in the induction assumption \ref{itm:inditem3}.
Hence we find that
\[
 \supp \rho_i \subset \overline{ V_i^{\tilde c_{|i|}} \cap \{z\mid 1-\sum_{k,|k|>|i|}\rho_k(z)>0\} }
 \subset 
 \overline{ V_i^{\tilde c_{|i|}} \cap \left(\tilde V_{{\rm gcd}(i,j)}^{c_{|{\rm gcd}(i,j)|}}\right)^c }
 \subset 
 V_i^{\tilde c_{|i|}} \cap \overline{ \left(\tilde V_{{\rm gcd}(i,j)}^{c_{|{\rm gcd}(i,j)|}}\right)^c }
\]
and similarly for $\supp \rho_j$, where we used in particular that the sets $V_i^c$ are closed.
It follows that
\begin{align*}
\supp \rho_i\cap \supp\rho_j
&\subset 
V_i^{\tilde c_{|i|}}
\cap
V_j^{\tilde c_{|j|}}
\cap 
\overline{ \left(\tilde V_{{\rm gcd}(i,j)}^{c_{|{\rm gcd}(i,j)|}}\right)^c }
\\&
\subset
V_{{\rm gcd}(i,j)}^{c_{|{\rm gcd}(i,j)|}}
\cap
\overline{ \left(\tilde V_{{\rm gcd}(i,j)}^{c_{|{\rm gcd}(i,j)|}}\right)^c }
=\emptyset
\end{align*}
where we have used \eqref{equ:intersectionprop}. This shows assertion \ref{itm:inditem4} for $n$.
In particular, it follows that for $i\neq j$, with $|i|=|j|=n$ the functions $ \rho_i$, $\rho_j$ have disjoint support and hence it is clear that \ref{itm:inditem2} holds. Finally, for \ref{itm:inditem3} it then suffices to note that on a neighborhood of $V_i^{c_{n}}$ on which $\chi_i\equiv 1$ with $|i|=n$
\[
 \sum_{k,|k|\geq n}\rho_k
 \geq 
 \rho_i + \sum_{k,|k|> n}\rho_k
 =
 \chi_i(1-\sum_{k,|k|> n}\rho_k) + \sum_{k,|k|> n}\rho_k
 =
 \chi_i+(1-\chi_i)\sum_{k,|k|> n}\rho_k
 =1+0=1.
\]
\end{proof}

\begin{lemma}
 Each $\rho_i$ is $T_i$-invariant.
\end{lemma}
\begin{proof}
Suppose that to the contrary we can find a configuration $z$ and a $t\in T_i$ such that $\rho_i(t\cdot z)\neq \rho_i(z)$. By replacing $z$ by $t^\alpha\cdot z$ for some $\alpha>0$ we may assume that in addition $\rho_i(t\cdot z), \rho_i(z)>0$ by continuity. In particular, we may assume that $z, t\cdot z\in\supp \rho_i$. But on $\supp \rho_i$ we may write by Lemma \ref{lem:rhoiprops}
\[
\rho_i
= 
\chi_i (1-\sum_{j, |j|>|i|}\rho_j)
=
\chi_i (1-\sum_{j, j < i}\rho_j)
\]
since for $j$ such that $i$, $j$ are non-ancestors $\rho_i$ and $\rho_j$ have disjoint support.
On the right-hand side each $\rho_j$ is $T_j$ invariant and hence also $T_i\subset T_j$ invariant. Since $\chi_i$ is $T_i$ invariant we conclude that $\rho_i(t\cdot z)= \rho_i(z)$, a contradiction. Hence the Lemma follows.
\end{proof}

Finally, we take for $U_i$ a small enough $T_i$-invariant neighborhood of $\supp \rho_i$, such that for $i$ and $j$ non-ancestors the assertion $U_i\cap U_j=\emptyset$ still holds. 

Summarizing, we have constructed a partition of unity $\{\rho_i\}$ subordinate to a cover $\{U_i\}$ such that:
\begin{itemize}
 \item $U_i\cap U_j=\emptyset$ unless $i\geq j$ or $j\geq i$.
 \item Each $U_i$ carries a free action of a torus $T_i$.
 \item The function $\rho_i$ is $T_i$ invariant.
 \item If $i>j$ there is a natural inclusion $T_j\to T_i$, compatible with the action of $T_i$ on $U_i\cap U_j$. 
\end{itemize}

\subsection{Configurations in the upper half-plane}
The configuration spaces of points in the upper half-plane and on the real line are described in a similar fashion. In more detail, let $\tau: z \mapsto \overline{z}$ be the complex conjugation, $A$ be a finite set and $\sigma$ be an involution on $A$. We shall represent $A$ as a disjoint union $A_+ \cup A_- \cup A_0$. Here $A_0$ is the fixed point set of $\sigma$, and $\sigma$ restricts to a bijection $A_+ \to A_-$. Furthermore, we shall fix a total order on $A_0$. Then, we consider injective maps $z: A \to \C$ which intertwine $\sigma$ and $\tau$. In particular, $A_0$ maps to the real line, and we require  its total order to be compatible with the natural increasing order on the real line. We shall also require that $A_+$ maps to the upper half-plane. Then, $A_-$ will automatically map to the lower half-plane.

Such a space of maps carries a free action of the group $G_2=\R_+ \ltimes \R$. The quotient by this action is the open configuration space ${\rm Conf}^{{\rm open}}_{A, \sigma}$. One can again choose an explicit section for the action of $G_2$ defined by the equations
$$
\sum_{a \in A}  z_a =0, \hskip 0.3cm \sum_{a \in A}  |z_a|^2 = 1.
$$
Note that the first equation now takes values in the reals.
The compactification is defined in a similar fashion to the case of ${\rm Conf}_A$.

The structure of manifold with corners on ${\rm Conf}_{A, \sigma}$ is again described by the charts associated to nested families. We now require that the nested families involved be $\sigma$-invariant. Note that such a nested family may contain subsets of two types: either we have a pair of subsets $B \subset A_+, \sigma(B) \subset A_-$ (subsets of type I) or a $\sigma$-invariant subset of $A$ (subsets of type II). If a subset of type II   contains some elements of $A_0$, these elements should form a sequence without gaps under the total order. The set $A_0$ together with the set of subsets of type II which do not contain elements of $A_0$ should be equipped with the total order consistent with the total order of $A_0$.

Corresponding to the two types of subsets there are two types of boundary strata: boundary strata of type I correspond to the collapse of a certain number of points $z_b$ for $b \in B$ in the upper half plane to their center of mass $\zeta_B$. At the same time, the points $z_{\sigma(b)} = \overline{z}_b$  are collapsing to $\overline{\zeta}_B$. Type II boundary strata correspond to the collapse of a group of points $z_b, \overline{z}_b, z_{c}=\overline{z}_c$ for $b, \sigma(b), c \in C$ to their  center of mass $\zeta_C \in \R$.

The chart $U_i \subset {\rm Conf}_{A, \sigma}$ carries  an action of the torus $T_i$ of dimension equal to the number of type I subsets in $i$. Indeed, for a  type I subset $B \in i$ the action of $S^1_B$ preserves the property $z_{\sigma(b)} = \overline{z}_b$ while this is not the case for type II subsets.

\section{Logarithmic weights}  \label{sec:log}

In this section we discuss the logarithmic weights following the suggestion by M. Kontsevich \cite{K2}.

Let $A$ be a finite set with the involution $\sigma$, and let $\Gamma$ be a finite oriented graph with the set of vertices $V\Gamma = A_+ \cup A_0$. It is called admissible if it doesn't have double edges (with the same orientation) and simple loops, and if no edges start at vertices $v\in A_0$. We denote $|A_+|=n, |A_0|=m$, and we let ${\rm Graph}_{n,m}$ be the set of admissible graphs with the vertex set $V\Gamma = A_+ \cup A_0$.

To such a graph, we associate a differential form $\Omega^{{\rm log}}_\Gamma$ of degree $|E \Gamma|$ on the open configuration space ${\rm Conf}^{{\rm open}}_{A, \sigma}$. This form is given by the formula
$$
\Omega^{{\rm log}}_\Gamma = \prod_{e \in E\Gamma} \,  \frac{1}{2\pi i} \, 
\log\left( \frac{ z_{s(e)}-z_{t(e)}}{\overline z_{s(e)} - {z}_{t(e)}} \right)
$$
where $s(e)$ and $t(e)$ are the source and the target of the edge $e$. 
%

Let $A=A_+ \cup A_-$ with $A_+=\{ s, t\}$ the set of two points, and choose a graph $\Gamma$ with the only edge starting at $s$ and ending at $t$. In this case, one can easily list all $\sigma$-invariant nested families:
$$
\begin{array}{lll}
\underline i_1 = \{ \{ s,t \}, \{ \sigma(s), \sigma(t)\} \} & \underline i_2 = \{ \{ s, \sigma(s)\} \} & \underline i_3 = \{ \{ t, \sigma(t) \} \} \\
\underline i_4 = \{ \{ s, \sigma(s)\} < \{ t, \sigma(t)\} \} & \underline i_5=\{ \{ t, \sigma(t)\} < \{ s, \sigma(s)\} \} . &
\end{array}
$$
Here the family $i_1$ consists of the pair of subsets of type I, all the other families consist of subsets of type II. It is easy to check that the 1-form $\Omega^{{\rm log}}_\Gamma$ extends smoothly to the boundary strata described by the charts $U_{i_k}$ with $k=2,3,4,5$, but it does not extend to the boundary stratum described by the chart $U_{i_1}$.  In general, these are type I boundary strata which pose problems, and below we analyze the behavior of $\Omega^{{\rm log}}_\Gamma$ on these strata.

\begin{prop} \label{prop:one_reg}
Let $\Gamma$ be an admissible graph. Consider a chart $U_i$, choose a vertex $B$ of
the tree defining $U_i$. Let $r_B \geq 0$ be the corresponding coordinate and $v_B$
be the fundamental vector field of the circle action on $U_i$.
Then, the form $\iota(v_B) \Omega_\Gamma^{{\rm log}}$ is regular in  $r_B$, and
the form $\Omega_\Gamma^{\rm log}$ admits a decomposition on $U_i$
$$
\Omega_\Gamma^{\rm log} = \frac{dr_B}{r_B} \, \wedge \alpha + {\rm terms \,\, regular \,\, in \,\, } r_B
$$
with $\alpha$ independent of $r_B$, $\iota(v_B) \alpha=0$ and $d\alpha=0$.
\end{prop}

\begin{proof}
In the form $\Omega_\Gamma^{\rm log}$, the
source of possible singularities with respect to the coordinate $r_B$
comes from the factors
$$
d \, \log(z-w) = \frac{d r_B}{r_B} + \, {\rm terms \,\, regular \,\, in \,\,} r_B,
$$
where both $z$ and $w$ belong to the set $B$. Hence, the form $\Omega_\Gamma^{\rm log}$
admits a decomposition
$$
\Omega_\Gamma^{\rm log} = \frac{d r_B}{r_B} \, \wedge \alpha + {\rm terms \,\, regular \,\, in \,\,} r_B,
$$
where $\alpha$ is independent of $r_B$. For the form $d \Omega_\Gamma^{\rm log}$, we obtain
$$
d \Omega_\Gamma^{\rm log} = - \frac{d r_B}{r_B} \, \wedge d \, \alpha + 
 \, {\rm terms \,\, regular \,\, in \,\,} r_B .
$$
Since $\Omega_\Gamma^{\rm log}$ is closed, we conclude that $d\alpha =0$.
Next, we compute
$$
\iota(v_B) \Omega_\Gamma^{\rm log} = - \frac{d r_B}{r_B} \, \wedge \iota(v_B) \alpha 
+ {\rm terms \,\, regular \,\, in \,\,} r_B .
$$
Hence, the form $\iota(v_B) \Omega_\Gamma^{\rm log}$ is regular in $r_B$ if and only if $\iota(v_B) \alpha =0$.

Consider the form
$$
\iota\left( \frac{\partial}{\partial r_B} \right) \, \iota(v_B) \Omega_\Gamma^{\rm log} = - r_B^{-1}  \iota(v_B) \alpha 
+ {\rm terms \,\, regular \,\, in \,\,} r_B .
$$
This form is regular in $r_B$ if and only if so is the form $\iota(v_B) \Omega_\Gamma^{\rm log}$.
Note that
$$
\frac{\partial}{\partial r_B} \wedge v_B = 
\frac{\partial}{\partial r_B} \wedge \left( v_B - i r_B \frac{\partial}{\partial r_B} \right),
$$
and hence
$$
\iota\left( \frac{\partial}{\partial r_B} \right)  \iota(v_B) \Omega_\Gamma^{\rm log} =
\iota\left( \frac{\partial}{\partial r_B} \right)  
\iota\left( v_B - i r_B \frac{\partial}{\partial r_B} \right) \Omega_\Gamma^{\rm log} .
$$

To complete the proof, we shall show that the form
$$
\iota\left( v_B - i r_B \frac{\partial}{\partial r_B} \right) \Omega_\Gamma^{\rm log}
$$
is regular in $r_B$. Indeed, for the points $z$ and $w$ in the set $B$
the circle action with generator $v_B$ maps $z-w \mapsto (z-w) \exp(i\phi)$, and
we have
$$
\iota\left( v_B - i r_B \frac{\partial}{\partial r_B} \right) \, d \log(z-w)  = i - i =0.
$$
For pairs of points $z$ and $w$ where at least one of the points does not belong
to the set $B$ both forms $\iota(v_a) \, d\log(z-w)$ and
$r_B \iota(\partial/\partial r_B) \, d\log(z-w)$ are proportional
to $r_B$. This factor cancels the denominator in $d r_B/r_B$, as required.
\end{proof}

\begin{rem}
The form $\alpha$ of Proposition \ref{prop:one_reg} is closed, $d \alpha=0$, and 
horizontal for the $S^1$-action generated by the vector field $v_B$, $\iota(v_B) \alpha=0$.
Hence, it is invariant under this circle action,
$$
L(v_B) \alpha = (d \iota(v_B) + \iota(v_B) d) \, \alpha = 0.
$$
\end{rem}

\begin{prop} \label{prop:many_reg}
Let $\Gamma$ be an admissible graph. 
Then, in every chart $U_i$ the form $\iota(\xi_i) \Omega_\Gamma^{{\rm log}}$
is regular.
\end{prop}

\begin{proof}
Let $r_1, \dots, r_k \geq 0$ be the coordinates corresponding
to the vertices of the tree defining the chart $U_i$. By Proposition \ref{prop:one_reg},
for every $B=1, \dots, k$ the form $\iota(v_B) \Omega_\Gamma^{\rm log}$ is regular
in $r_B$. Hence, the form
$$
\iota(\xi_i) \Omega_\Gamma^{\rm log} = 
\iota( v_1 \wedge \dots \wedge v_B \wedge \dots \wedge v_k) \Omega_\Gamma^{\rm log}
$$
is also regular in $r_B$. Since this argument applies to all $B = 1, \dots, k$ we conclude that
the form $\iota(\xi_i) \Omega_\Gamma^{\rm log}$ is regular on $U_i$.
\end{proof}

\begin{thm}\label{thm:formregular}
Let $\Gamma \in {\rm Graph}_{n,m}$ be an admissible graph such that $| E \Gamma|= 2n+m-2$.
Then, the form $\Omega_\Gamma^{{\rm log}}$ is regular.
\end{thm}

\begin{proof}
Let $U_i$ be a chart, $\theta_i \in \Omega(U_i, \mathfrak{t}_i)$ be a connection 1-form for
the $T_i$-action on $U_i$, and ${\rm vol}_i \in \Omega^{{\rm dim} \, \mathfrak{t}_i}(U_i)$ the volume form
on the orbits of the $T_i$-action, 
$$
{\rm vol}_i = \theta_i^k \, \wedge \dots \wedge \, \theta_i^1 .
$$
Since the $T_i$-action on $U_i$ is free and $\Omega_\Gamma^{{\rm log}}$ is a top degree form, its restriction to $U_i$ is given by
$$
\Omega_\Gamma^{{\rm log}} |_{U_i} = {\rm vol}_i \wedge \iota(\xi_i) \Omega_\Gamma^{{\rm log}}.
$$
By Proposition \ref{prop:many_reg},
the right-hand side is regular. We conclude that $\Omega_\Gamma^{{\rm log}}$ is regular
in all charts $U_i$. Hence, it is a regular form, as required.
\end{proof}

\begin{prop}
Let $\Gamma \in {\rm Graph}_{n,m}$ be an admissible graph such that $| E \Gamma|= 2n+m- 3$.
Then, in every chart $U_i$ the form $\Omega_\Gamma^{{\rm log}}$ admits a decomposition
$$
\Omega_\Gamma^{{\rm log}} = \sum_{B=1}^k \, \frac{d r_B}{ r_B} \, \wedge \alpha_B + {\rm regular \,\, terms} ,
$$
where $\iota(\xi_i) \alpha_B =0$ and $\iota(\xi_i) d\alpha_B =0$ for all $B$.
\end{prop}

\begin{proof}
For $k=1$, the statement follows from Proposition \ref{prop:one_reg}.

For $k\geq 2$, denote by
$$
\xi_i^B = v_1 \wedge \dots \wedge \widehat{v_B} \wedge \dots \wedge v_k
$$
the polyvectors of degree $k-1$, and introduce the dual $(k-1)$-forms
$$
{\rm vol}_i^B = \theta_i^k \wedge \dots \wedge \widehat{\theta_i^B} \wedge \dots \wedge \theta_i^1 .
$$

Since the form $\Omega_\Gamma^{\rm log}$ is of top degree minus one, it admits a decomposition
$$
\Omega_\Gamma^{\rm log} = \sum_{B=1}^k {\rm vol}_i^B \wedge \iota(\xi_i^B) \Omega_\Gamma^{\rm log}
- (k-1) \, {\rm vol}_i \wedge \iota(\xi_i)  \Omega_\Gamma^{\rm log} .
$$
By Proposition \ref{prop:many_reg}, the last term on the right-hand side is regular. Consider one of the
terms in the sum,
$$
{\rm vol}_i^B \wedge \iota( v_1 \wedge \dots \wedge \widehat{v_B} \wedge \dots \wedge v_k) \, \Omega_\Gamma^{\rm log} .
$$
Applying Proposition \ref{prop:one_reg} to all the indices $b \neq B$ we see that this form is regular in $r_b, b \neq B$.
Using the same Proposition for $b=B$ we obtain
$$
{\rm vol}_i^B \wedge \iota( v_1 \wedge \dots \wedge \widehat{v_B} \wedge \dots \wedge v_k) 
\left( \frac{d r_B}{r_B} \wedge \alpha + {\rm terms\,\, regular \,\, in \,\,} r_B \right),
$$
where $\iota(v_B) \alpha =0$ and $d\alpha=0$. Using notation
$$
\alpha_B = {\rm vol}_i^B \wedge \iota( v_1 \wedge \dots \wedge \widehat{v_B} \wedge \dots \wedge v_k)  \, \alpha
$$
we obtain the required decomposition of $\Omega_\Gamma^{\rm log}$. By construction,
$$
\iota(\xi_i) \alpha_B = \iota( v_1 \wedge \dots \wedge v_B \wedge \dots \wedge v_k) \, \alpha =0.
$$
Furthermore, $L(v_B) \alpha_B =0$ since  $L(v_B) \theta_i^b=0$ for all $b$ and 
$L(v_B) \alpha=0$ by Proposition \ref{prop:one_reg}. Hence,
$\iota(v_B) d\alpha_B = L(v_B) \alpha_B - d \iota(v_B) \alpha_B =0$ and
$$
\iota(\xi_i) d \alpha_B=\iota(v_1 \wedge \dots \wedge v_B \wedge \dots \wedge v_k) \, d \alpha_B=0,
$$
 as required.


\end{proof}


Let $\Gamma \in {\rm Graph}_{n,m}$ be an admissible graph, $U_i$ be a chart and $B \in i$ be a vertex of the tree defining
$\Gamma$. We shall use the notation $\partial_B \Gamma = \Gamma' \cup \Gamma''$, where 
$\Gamma' \subset \Gamma$ is the subgraph corresponding to the vertices which belong to the subset $B$, and $\Gamma''$ is the graph obtained from $\Gamma$ by contracting the subgraph $\Gamma'$.

\begin{prop}  \label{prop:restrict}
Let $\Gamma \in {\rm Graph}_{n,m}$ be an admissible graph.
Choose a chart $U_i$ and a vertex $B$ of the tree defining $U_i$. Let $ \partial_B U_i$ be the co-dimension
one stratum of the boundary of $U_i$ corresponding to $B$ and denote  $\partial_B \Gamma = \Gamma' \cup \Gamma''$.
Then, 
$$
\iota(v_B) \Omega^{\rm log}_\Gamma |_{\partial_B U_i} = \iota(v_B) \Omega^{\rm log}_{\Gamma'} |_{\partial_B U_i} 
\wedge \Omega^{\rm log}_{\Gamma''} |_{\partial_B U_i} .
$$

%
\end{prop}

\begin{proof}
By Proposition \ref{prop:one_reg}, the forms $\iota(v_B) \Omega^{\rm log}_\Gamma$ and $\iota(v_B) \Omega^{\rm log}_{\Gamma'}$
are regular in $r_B$, and there is a decomposition
$$
\Omega^{\rm log}_{\Gamma'} = \frac{d r_B}{r_B} \wedge \beta + \gamma,
$$
where $\beta$ is independent of $r_B$,  $\gamma$ is regular in $r_B$ and $\iota(v_B) \beta =0$

The form $\Omega^{\rm log}_{\Gamma''}$ is also regular in $\rho_B$. Moreover, it is a product of 1-forms
$d \log(z-w)$ where at most one of the points $z$ and $w$ belongs to the collapsing set labeled by $a$. Hence,
$$
\iota(v_B) \,\,d \log(z-w) = r_B  \cdot \, {\rm a \,\, function \,\, regular \,\, in \,\,} r_B
$$
and $\iota(v_B) \Omega^{\rm log}_{\Gamma''} = r_B \alpha$, where $\alpha$ is a form regular in $r_B$. 

We compute
$$
\begin{array}{lll}
\iota(v_B) \Omega^{\rm log}_\Gamma & = & 
\iota(v_B) \Omega^{\rm log}_{\Gamma'} \wedge \Omega^{\rm log}_{\Gamma''} \pm 
\Omega^{\rm log}_{\Gamma'} \wedge \iota(v_B)  \Omega^{\rm log}_{\Gamma''}  \\
& = & \iota(v_B) \Omega^{\rm log}_{\Gamma'} \wedge \Omega^{\rm log}_{\Gamma''}  \pm
\left(  \frac{d r_B}{r_B} \wedge \beta + \gamma \right) \wedge \, r_B  \alpha \\
& = & 
\iota(v_B) \Omega^{\rm log}_{\Gamma'} \wedge \Omega^{\rm log}_{\Gamma''}  \pm
d r_B \wedge \beta  \wedge  \alpha + r_B \gamma \wedge  \alpha.
\end{array}
$$
The last two terms in the last line are proportional to $r_B$ and $d r_B$. Hence, they vanish
when restricted to $\partial_B U_i$, as required.

\end{proof}

\begin{thm}\label{thm:vanishingproperty}
Let $\Gamma \in {\rm Graph}_{n,m}$ with $|E\Gamma| = 2n +m-3$, 
$U_i$ be a coordinate chart on the corresponding configuration space
and $B$ be a vertex of the tree defining $U_i$. If $\partial_B U_i$ is a type I (interior) boundary stratum,
then
$$
{\rm Reg}_B \, (\Omega^{\rm log}_\Gamma) =
\begin{cases}
\frac{d\phi}{2\pi} \wedge \Omega^{\rm log}_{\Gamma''} |_{\partial_B U_i} 
& \text{ if $V\Gamma' = \{ x, y\}$ and $E\Gamma'$ is an edge connecting $x$ and $y$} \, , \\
0 & \text{otherwise} \, ,
\end{cases}
$$
where $\phi$ is the natural coordinate on ${\rm Conf}_2 \cong S^1$.
If $\partial_a U_i$ is a type II boundary stratum, then
$$
{\rm Reg}_B \, (\Omega^{\rm log}_\Gamma) = \Omega^{\rm log}_{\Gamma'} |_{\partial_B U_i} 
\wedge \Omega^{\rm log}_{\Gamma''} |_{\partial_B U_i} \, .
$$
\end{thm}

\begin{proof}
In the case of type II boundary strata, the forms $\Omega^{\rm log}_\Gamma, \Omega^{\rm log}_{\Gamma'}$
and $\Omega^{\rm log}_{\Gamma''}$ are regular in $r_B$. Hence, 
$$
{\rm Reg}_B \, (\Omega^{\rm log}_\Gamma) =  \Omega^{\rm log}_{\Gamma} |_{\partial_B U_i}
= \Omega^{\rm log}_{\Gamma'} |_{\partial_B U_i} 
\wedge \Omega^{\rm log}_{\Gamma''} |_{\partial_B U_i} \, .
$$

For a type I boundary stratum with $V \Gamma' = \{ x, y\}$ and $E\Gamma'$ being a single edge connecting
$y$ to $x$ denote coordinates of the corresponding points in the upper half plane by $z$ and $w$. Then,
$$
\Omega^{\rm log}_{\Gamma'} = \frac{1}{2\pi i} \,\, d \, \log\left( \frac{z-w}{z - \overline{w}} \right) .
$$
Using the parametrization $z= \zeta+ r_B e^{i\phi}, w = \zeta - r_B e^{i\phi}$ we get
$$
\iota(v_B) \Omega^{\rm log}_{\Gamma'} = \frac{1}{2\pi}  - \frac{2 r_B \sin(\phi)}{2\pi i(\zeta - \overline{\zeta} + 2r_B \cos(\phi))} .
$$
By Proposition \ref{prop:restrict}, we have
$$
\iota(v_B) \Omega^{\rm log}_\Gamma |_{\partial_B U_i} = \iota(v_B) \Omega^{\rm log}_{\Gamma'} |_{\partial_B U_i} 
\wedge \Omega^{\rm log}_{\Gamma''} |_{\partial_B U_i} = \frac{1}{2\pi} \,  \Omega^{\rm log}_{\Gamma''} |_{\partial_B U_i}.
$$
Choosing the connection for the $S^1$-action $\theta = d \phi$, we arrive at
$$
{\rm Reg}_B \, (\Omega^{\rm log}_\Gamma) = \theta \wedge \iota(v_B) \Omega^{\rm log}_\Gamma |_{\partial_B U_i} =
\frac{d\phi}{2\pi} \, \wedge \, \Omega^{\rm log}_{\Gamma''} |_{\partial_B U_i}.
$$

Finally, consider the case of a type I boundary stratum with $V\Gamma'$ containing $k>2$ vertices.
For each edge $e \in E\Gamma'$, we decompose the corresponding 1-form as
$$
\frac{1}{2\pi i} \,\, d \, \log\left( \frac{z-w}{z - \overline{w}} \right) = \frac{1}{2\pi i} \, d \, \log(z-w)
- \frac{1}{2\pi i} \, d \,  \log(z-\overline{w}).
$$
The logarithmic form $\Omega^{\rm log}_\Gamma$ is represented as a sum of terms 
$$
\Omega^{\rm log}_{\Gamma'} = \sum_j \alpha_j \wedge \beta_j,
$$
where $\alpha_k$'s are wedge products of holomorphic 1-forms of the type $d \, \log(z-w)$,
and $\beta_k$'s are wedge products of 1-forms of the type $d\, \log(z-\overline{w})$.
By repeating the argument in the proof of Proposition \ref{prop:restrict}, we obtain
$$
\iota(v_B) \Omega^{\rm log}_{\Gamma'}|_{\partial_B U_i} = \sum_j \, (\iota(v_B) \alpha_j)|_{\partial_B U_i}
\wedge \beta_j |_{\partial_B U_i}.
$$
%

Consider the factor $\iota(v_B)\alpha_j$. It is invariant under the action of $S^1_B$,
and it is horizontal since $\iota(v_B)(\iota(v_B) \alpha_j) = \iota(v_B)^2 \alpha_j =0$. Hence, $\iota(v_B) \alpha_j$
is basic and descends to the quotient ${\rm FM}_k={\rm Conf}_k/S^1$. The quotient space ${\rm FM}_k$ has a complex structure induced
by the one of $\mathbb{C}^k$, its complex dimension is $k-2$ and the real dimension is $2(k-2)$. 
The form $\iota(v_B) \alpha_j$ is holomorphic and hence its degree is at most $k-2$, its top degree part
vanishes. Since $|E\Gamma| = 2n+m -3$, the form
$$
{\rm Reg}_B \, (\Omega^{\rm log}_{\Gamma}) = \theta \wedge \iota(v_B) \Omega^{\rm log}_{\Gamma}|_{\partial_B U_i}
$$
is of top degree on $\partial_B U_i$. Hence, it vanishes, as required.

\end{proof}

\section{The formality morphism}   \label{sec:final}

In this Section, we construct the logarithmic formality morphism $\mU^{{\rm log}}$ and prove its globalization property.

\subsection{Recollection: Kontsevich's formality morphism}
Kontsevich's construction \cite{K1} of the $L_\infty$-quasi-isomorphism
\[
 \mU\colon \Tpoly M\to \Dpoly M
\]
between the graded Lie algebra of multivector fields $\Tpoly M$ and the differential graded Lie algebra of multi-differential operators $\Dpoly M$ on a smooth manifold $M$ proceeds in two steps. First, for $M=\R^d$ one defines the $n$-th component of the formality morphism as 
\[
 \mU_n(\gamma_1,\dots, \gamma_n) 
 =
 \sum_{\Gamma\in {\rm Graph}_{n,m}}
 \varpi_\Gamma D_\Gamma(\gamma_1,\dots, \gamma_n)
\]
where $\gamma_1,\dots, \gamma_n\in \Tpoly$, $m=|\gamma_1|+\cdots +|\gamma_n|+2-n$, the sum is over the set of admissible graphs ${\rm Graph}_{n,m}$ and the $D_\Gamma(\dots)$ is an $m$-differential operator naturally associated to such a graph.
One fixes an arbitrary ordering on the set of edges $E\Gamma$ for each admissible graph $\Gamma$. 

The coefficients $\varpi_\Gamma\in \mathbb{R}$ are defined through configuration space integrals.
\[
 \varpi_\Gamma = \int_{{\rm Conf}_{n,m}} \Omega_\Gamma
\]
The top degree form in the integrand is defined by \eqref{eq:1}.
Here the product is taken in the order fixed on the set of edges, thus resolving the sign ambiguity.
As mentioned in the introduction, the statement that $\mU$ is an $L_\infty$-morphism then translates into a set of quadratic identities to be satisfied by the coeffcients $\varpi_\Gamma$. It turns out that these quadratic equations are exactly the quadratic equations obtained by using Stokes' formula for graphs $\Gamma$ with $|E\Gamma| = 2n+m -3$
\begin{equation}
 \label{equ:KStokes}
0=\int_{{\rm Conf}_{n,m}} d\Omega_\Gamma
 =
 \int_{\partial \, {\rm Conf}_{n,m}} \Omega_\Gamma
 =
 \sum_{i} \int_{\partial_i \, {\rm Conf}_{n,m}} \Omega_\Gamma 
\end{equation}
provided that the following vanishing property holds:

\vskip 3mm 

{\bf Vanishing Lemma, \cite[Lemma 6.6]{K1}:} The contribution of the type I boundary strata 
on the right-hand side of \eqref{equ:KStokes} vanishes if more than 2 points collapse. 

\vskip 3mm

Furthermore, one checks that the component $\mU_1$ is exactly the Hochschild-Kostant-Rosenberg morphism, and hence $\mU$ is indeed a quasi-isomorphism, completing the first step of the construction.

In the second step, one globalizes the formality result from $\R^d$ to general smooth manifolds $M$. As shown in \cite{K1} a formality morphism given by universal formulas (i.~e., expressible through linear combinations of admissible graphs) can be globalized if it satisfies the following vanishing properties:\footnote{In fact, there are five properties in \cite{K1} to be satisfied. However, the others are trivially true for a formality morphism given by universal formulas such that $\mU_1$ is the Hochschild-Kostant-Rosenberg morphism.}

\vskip 3mm

{\bf Kontsevich globalization conditions \cite[section 7]{K1}:}\nopagebreak
\begin{enumerate}
 \item For any vector fields $\xi_1,\xi_2$ we have $\mU_2(\xi_1, \xi_2)=0$.
 \item For any linear vector field $\xi$ and any multivector fields $\gamma_2,\dots,\gamma_n$ we have $\mU_n(\xi,\gamma_1, \dots, \gamma_n)=0$.
\end{enumerate}

\vskip 3mm

\subsection{Logarithmic formality}
One can construct a new formality morphism $\mU^{\rm log}$ using Kontsevich's technique outlined in the preceding subsection, but using the logarithmic weight form 
\[
\Omega^{{\rm log}}_\Gamma
=
\prod_{(i,j)\in E\Gamma} \frac{1}{2\pi i} \, d\log\left( \frac{z_i-z_j}{\bar z_i-z_j} \right)
\]
in place of the form $\Omega_\Gamma$ above. Concretely, for an admissible graph $\Gamma$ with $|E\Gamma| = 2n + m -2$ we denote by
\begin{equation}\label{equ:varpilog}
\varpi^{{\rm log}}_\Gamma = \int_{{\rm Conf}_{n,m}} \, \Omega^{{\rm log}}_\Gamma
\end{equation}
the corresponding logarithmic weights. 
By Theorem \ref{thm:formregular}, the integrand on the right-hand side has no singularities on the boundary, and the integral is well-defined.
The numbers $\varpi^{{\rm log}}_\Gamma$ satisfy quadratic identities obtained by applying the regularized Stokes Theorem, i.~e. Theorem \ref{thm:regStokes} to the differential forms of top minus one degree $\Omega^{{\rm log}}_\Gamma$ for admissible graphs $\Gamma$ such that $|E\Gamma| = 2n + m -3$:
\[
 0 = \int_{{\rm Conf}_{n,m}} d\Omega^{{\rm log}}_\Gamma
 =
 \int_{\partial \, {\rm Conf}_{n,m}} {\rm Reg}\, \Omega^{{\rm log}}_\Gamma.
\]
By Theorem \ref{thm:vanishingproperty}, the Kontsevich vanishing property is still satisfied, and hence we can define an $L_\infty$.morphism 
\[
 \mU^{\rm log} \colon \Tpoly \to \Dpoly
\]
by setting
\[
 \mU^{\rm log}_n(\gamma_1,\dots, \gamma_n) 
 =
 \sum_{\Gamma\in {\rm Graph}_{n,m}}
 \varpi^{\rm log}_\Gamma D_\Gamma(\gamma_1,\dots, \gamma_n).
\]

Note that the logarithmic one form associated to an edge agrees with the Kontsevich one form if one of the two arguments is real. In particular this means that for graphs $\Gamma$ with a single type I vertex $\Omega_\Gamma=\Omega_\Gamma^{\rm log}$ and hence $\varpi^{\rm log}_\Gamma=\varpi_\Gamma$.
It follows that $\mU^{\rm log}_1=\mU_1$ agrees with the Hochschild-Kostant-Rosenberg morphism. This means in particular that the $L_\infty$-morphism $\mU^{\rm log}$ is a quasi-isomorphism.

Furthermore, the morphism $\mU^{\rm log}$ can be globalized according to \cite{K1}, as the following proposition shows.

\begin{prop}
 The formality morphism $\mU^{\rm log}$ satisfies the Kontsevich globalization conditions, i.~e., 
 \begin{itemize}
  \item For any two vector fields $\xi_1,\xi_2$ we have $\mU^{\rm log}_2(\xi_1, \xi_2)=0$.
 \item For any linear vector field $\xi$ and any multivector fields $\gamma_2,\dots,\gamma_n$: $\mU^{\rm log}_n(\xi,\gamma_1, \dots, \gamma_n)=0$.
  \end{itemize}
\end{prop}
\begin{proof}
Assume that the admissible graph $\Gamma\in {\rm Graph}_{n,m}$ with $n\geq 2$ contains a univalent vertex with associated coordinate $z$. The corresponding configuration space integral involves an integration over $z$ ranging over a two dimensional space, while there is only one 1-form which depends on $z$. Hence the integral is zero by degree reasons, i.~e.,
$\varpi^{\rm log}_\Gamma=0$ . 

Next suppose that $\Gamma$ contains a vertex with exactly one incoming and one outgoing edge, so it locally looks like this:
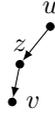
\[
\begin{tikzpicture}
\node[int, label={$u$}] (u) at (.5,1.5) {};
\node[int, label={$z$}] (z) at (.1,1) {};
\node[int, label=0:{$v$}] (v) at (0,.5) {};
\draw(-1,0)--(1,0);
\draw[-latex] (u) edge (z) (z) edge (v);.
\end{tikzpicture}
\]
Here possibly $u=v$. We claim that in this case $\varpi^{\rm log}_\Gamma=0$.
For fixed $u$, $v$, the part of the integral involving $z$ reads
\begin{align*}
  \int_{\bbH\setminus\{u,v\}}
  d\log\frac{u-z}{\bar u-z} \, \wedge \, d\log\frac{z-v}{\bar z -v}
  &=
   \int_{\bbH\setminus\{u,v\}}
   d\left(
  \log(\bar z -v) \, \wedge \, d\log\frac{u-z}{\bar u-z} 
  \right)
  \\&=
  -2\pi i \log(\bar u -v)
  +
  \int_{-\infty}^\infty
  \log(x -v) \left( \frac{1}{x-u}-\frac{1}{x-\bar u}\right) dx
  \\&=
  -2\pi i \log(\bar u -v)
  +
  2\pi i \log(\bar u -v)
  =0\, .
\end{align*}
Here we used Stokes' formula, and then evaluated the integral over $\R$ by closing the contour in the lower half-plane. Then, only the pole at $\bar u$ contributes, and  the result follows. 

Overall, we have shown that $\varpi^{\rm log}_\Gamma=0$ if the graph $\Gamma$ has either univalent vertices or bivalent vertices with one incoming and one outgoing edge.
But any graph that could possibly contribute to $\mU^{\rm log}_2(\xi_1, \xi_2)$ or $\mU^{\rm log}_n(\xi,\gamma_1, \dots, \gamma_n)$ has one of such features and hence the proposition follows.
\end{proof}

Summarizing, we have shown the following theorem:
\begin{thm}   \label{thm:final}
 The configuration space integrals \eqref{equ:varpilog} exist and define coefficients of a stable formality morphism $\mU^{\rm log}$, that satisfies the Kontsevich globalization conditions.
\end{thm}

\begin{rem}
The characteristic class (see \cite{thomaschar}), also known as the Duflo function of the formality morphism $\mU^{{\rm log}}$ has been computed in \cite{merkulov_exotic} and \cite{carloexplicit} and is equal to
 \[
\exp\left(\sum_{n\geq 2} \frac{\zeta(n)}{n(2\pi i)^n} x^n\right)\, .
 \]
\end{rem}

\begin{rem}
Note that if $\Gamma$ is an admissible graph with a type I vertex with no outgoing edges and at least two vertices, then $\varpi^{\rm log}_\Gamma=0$. Indeed, if we denote by $z$ the point in the upper half-plane defined by the vertex, then the integrand in \eqref{equ:varpilog} has no term involving $d\bar z$. 
In particular, it follows that the result \cite{shwheel} on vanishing of the wheel graphs is trivially true for $\mU^{\rm log}$.
\end{rem}

\begin{rem}
 The construction \cite{AT} of the first and third authors may be generalized to obtain a Drinfeld associator corresponding to the logarithmic propagator, which turns out to be the Knizhnik-Zamolodchikov associator. We leave the details to elsewhere.
\end{rem}

\end{document}